\documentclass[11pt]{article}
\usepackage[margin=1in]{geometry} 
\usepackage{amssymb,amsfonts,amsmath,bbm,mathrsfs,stmaryrd,mathtools}
\usepackage{xcolor}
\usepackage{url}

\usepackage{extarrows}

\usepackage[shortlabels]{enumitem}
\usepackage{tensor}

\usepackage{xr}
\usepackage[T1]{fontenc}

\usepackage[colorlinks,
linkcolor=black!75!red,
citecolor=blue,
pdftitle={},
pdfproducer={pdfLaTeX},
pdfpagemode=None,
bookmarksopen=true,
bookmarksnumbered=true,
backref=page]{hyperref}

\usepackage{tikz}
\usetikzlibrary{arrows,calc,decorations.pathreplacing,decorations.markings,intersections,shapes.geometric,through,fit,shapes.symbols,positioning,decorations.pathmorphing}

\usepackage{braket}

\usepackage[amsmath,thmmarks,hyperref]{ntheorem}
\usepackage{cleveref}

\creflabelformat{enumi}{#2#1#3}

\crefname{section}{Section}{Sections}
\crefformat{section}{#2Section~#1#3} 
\Crefformat{section}{#2Section~#1#3} 

\crefname{subsection}{\S}{\S\S}
\AtBeginDocument{%
  \crefformat{subsection}{#2\S#1#3}%
  \Crefformat{subsection}{#2\S#1#3}%
}

\crefname{subsubsection}{\S}{\S\S}
\AtBeginDocument{%
  \crefformat{subsubsection}{#2\S#1#3}%
  \Crefformat{subsubsection}{#2\S#1#3}%
}

\theoremstyle{plain}

\newtheorem{lemma}{Lemma}[section]
\newtheorem{proposition}[lemma]{Proposition}
\newtheorem{corollary}[lemma]{Corollary}
\newtheorem{theorem}[lemma]{Theorem}

\theoremstyle{plain}
\theoremnumbering{Alph}
\newtheorem{theoremN}{Theorem}

\theoremstyle{plain}
\theorembodyfont{\upshape}
\theoremsymbol{\ensuremath{\blacklozenge}}

\newtheorem{definition}[lemma]{Definition}
\newtheorem{example}[lemma]{Example}

\newtheorem{remark}[lemma]{Remark}
\newtheorem{remarks}[lemma]{Remarks}

\crefname{definition}{definition}{definitions}
\crefformat{definition}{#2definition~#1#3} 
\Crefformat{definition}{#2Definition~#1#3} 

\crefname{ex}{example}{examples}
\crefformat{example}{#2example~#1#3} 
\Crefformat{example}{#2Example~#1#3} 

\crefname{exs}{example}{examples}
\crefformat{examples}{#2example~#1#3} 
\Crefformat{examples}{#2Example~#1#3} 

\crefname{remark}{remark}{remarks}
\crefformat{remark}{#2remark~#1#3} 
\Crefformat{remark}{#2Remark~#1#3} 

\crefname{remarks}{remark}{remarks}
\crefformat{remarks}{#2remark~#1#3} 
\Crefformat{remarks}{#2Remark~#1#3} 

\crefname{convention}{convention}{conventions}
\crefformat{convention}{#2convention~#1#3} 
\Crefformat{convention}{#2Convention~#1#3} 

\crefname{notation}{notation}{notations}
\crefformat{notation}{#2notation~#1#3} 
\Crefformat{notation}{#2Notation~#1#3} 

\crefname{table}{table}{tables}
\crefformat{table}{#2table~#1#3} 
\Crefformat{table}{#2Table~#1#3}

\crefname{lemma}{lemma}{lemmas}
\crefformat{lemma}{#2lemma~#1#3} 
\Crefformat{lemma}{#2Lemma~#1#3} 

\crefname{proposition}{proposition}{propositions}
\crefformat{proposition}{#2proposition~#1#3} 
\Crefformat{proposition}{#2Proposition~#1#3} 

\crefname{propositionN}{proposition}{propositions}
\crefformat{propositionN}{#2proposition~#1#3} 
\Crefformat{propositionN}{#2Proposition~#1#3} 

\crefname{corollary}{corollary}{corollaries}
\crefformat{corollary}{#2corollary~#1#3} 
\Crefformat{corollary}{#2Corollary~#1#3} 

\crefname{corollaryN}{corollary}{corollaries}
\crefformat{corollaryN}{#2corollary~#1#3} 
\Crefformat{corollaryN}{#2Corollary~#1#3} 

\crefname{theorem}{theorem}{theorems}
\crefformat{theorem}{#2theorem~#1#3} 
\Crefformat{theorem}{#2Theorem~#1#3} 

\crefname{theoremN}{theorem}{theorems}
\crefformat{theoremN}{#2theorem~#1#3} 
\Crefformat{theoremN}{#2Theorem~#1#3} 

\crefname{enumi}{}{}
\crefformat{enumi}{#2#1#3}
\Crefformat{enumi}{#2#1#3}

\crefname{assumption}{assumption}{Assumptions}
\crefformat{assumption}{#2assumption~#1#3} 
\Crefformat{assumption}{#2Assumption~#1#3} 

\crefname{construction}{construction}{Constructions}
\crefformat{construction}{#2construction~#1#3} 
\Crefformat{construction}{#2Construction~#1#3} 

\crefname{equation}{}{}
\crefformat{equation}{(#2#1#3)} 
\Crefformat{equation}{(#2#1#3)}

\numberwithin{equation}{section}
\renewcommand{\theequation}{\thesection-\arabic{equation}}

\theoremstyle{nonumberplain}
\theoremsymbol{\ensuremath{\blacksquare}}

\newtheorem{proof}{Proof}
\newcommand\pf[1]{\newtheorem{#1}{Proof of \Cref{#1}}}

\newcommand\wh[1]{{\widehat{#1}}}

\newcommand\bC{{\mathbb C}}

\newcommand\bG{{\mathbb G}}

\newcommand\bR{{\mathbb R}}
\newcommand\bS{{\mathbb S}}

\newcommand\bZ{{\mathbb Z}}

\newcommand\cC{{\mathcal C}}

\newcommand\cE{{\mathcal E}}

\newcommand\cM{{\mathcal M}}

\newcommand\ol{\overline}

\DeclareMathOperator{\id}{id}

\DeclareMathOperator{\spn}{span}

\DeclareMathOperator{\wo}{\widehat{\otimes}}
\DeclareMathOperator{\rk}{\mathrm{rk}}

\newcommand\numberthis{\addtocounter{equation}{1}\tag{\theequation}}

\newcommand{\cat}[1]{\textsc{#1}}

\newcommand\spr[1]{\cite[\href{https://stacks.math.columbia.edu/tag/#1}{Tag {#1}}]{stacks-project}}
\newcommand{\qedhere}{\mbox{}\hfill\ensuremath{\blacksquare}}

\renewcommand{\square}{\mathrel{\Box}}

\newcommand{\xrightarrowdbl}[2][]{%
  \xrightarrow[#1]{#2}\mathrel{\mkern-14mu}\rightarrow
}

\title{Small Banach bundles and modules}
\author{Alexandru Chirvasitu}

\begin{document}

\date{}

\newcommand{\Addresses}{{
  \bigskip
  \footnotesize

  \textsc{Department of Mathematics, University at Buffalo}
  \par\nopagebreak
  \textsc{Buffalo, NY 14260-2900, USA}  
  \par\nopagebreak
  \textit{E-mail address}: \texttt{achirvas@buffalo.edu}


}}

\maketitle

\begin{abstract}
  We characterize those (continuously-normed) Banach bundles $\mathcal{E}\to X$ with compact Hausdorff base whose spaces $\Gamma(\mathcal{E})$ of global continuous sections are topologically finitely-generated over the function algebra $C(X)$, answering a question of I. Gogi\'c's and extending analogous work for metrizable $X$. Conditions equivalent to topological finite generation include: (a) the requirement that $\mathcal{E}$ be locally trivial and of finite type along locally closed and relatively $F_{\sigma}$ strata in a finite stratification of $X$; (b) the decomposability of arbitrary elements in $\ell^p(\Gamma(\mathcal{E}))$, $1\le p<\infty$ as sums of $\le N$ products in $\ell^p(C(X))\cdot \Gamma(\mathcal{E})$ for some fixed $N$; (c) the analogous decomposability requirement for maximal Banach-module tensor products $F\widehat{\otimes}_{C(X)}\Gamma(\mathcal{E})$ or (d) equivalently, only for $F=\ell^1(C(X))$.
\end{abstract}


\vspace{.5cm}



\section*{Introduction}

The initial motivation for the present note was the desire to apply the results of \cite{gog_top-fg} in the context of \cite{2402.08121v1}. The former gives in \cite[Theorems 1.1 and 3.6]{gog_top-fg}, for a {\it Hilbert bundle} \cite[Definition 13.5]{fd_bdl-1} $\cE\xrightarrowdbl{}X$ over a compact metrizable base space, a number of conditions equivalent to the topological finite generation of the section space $\Gamma(\cE)$ as a {\it Hilbert module} (\cite[p.4]{lnc_hilb}, \cite[Definition 15.1.5]{wo}) over the complex function $C^*$-algebra $C(X)$ (meaning, as is presumably guessable, that $\Gamma(\cE)$ has a finitely-generated dense purely algebraic $C(X)$-submodule). 

Metrizability is used crucially at several junctures in said results. The bundles and modules of relevance to \cite{2402.08121v1}, on the other hand, are over arbitrary compact (Hausdorff) base spaces, being attached to the initial data consisting of an action $\bG\times X\to X$ of a compact group on a compact Hausdorff space and an irreducible $\bG$-representation $\alpha$: in Hilbert-module guise, the object in question is the {\it $\alpha$-isotypic component} \cite[Definition 4.21]{hm5} $C(X)^{\alpha}$, regarded as a module over the invariant subalgebra $C(X)^{\cat{triv}}\cong C(X/\bG)$. The base space is thus (the possibly non-metrizable) $X/\bG$. 

In teasing out the precise role metrizability plays in characterizing topological finite generation, it transpires that the {\it Hilbert}-bundle/module framework can be expanded to accommodate {\it Banach} bundles (in the sense of \cite[Definition 13.4]{fd_bdl-1}) and modules respectively. This also ties in with \cite[Problem 3.11]{gog_top-fg}, which proposes extending the results under discussion, to whatever extent possible, to Banach modules attached to Banach bundles. The remainder of these introductory remarks elaborate on some of the specifics, taking some of the background for granted; more of it is recalled in the prefatory remarks in \Cref{se:smlbdl}. 

To fix ideas, Banach and Hilbert spaces are assumed complex. In the sequel, unqualified {\it Banach (Hilbert) bundles} over spaces $X$ as as in \cite[Definitions 13.4 and 13.5]{fd_bdl-1} (and \cite[\S\S 1 and 2]{fell_ext}, etc.):
\begin{itemize}[wide]
\item continuous open surjections $\cE\xrightarrowdbl{\pi}X$ (between typically Hausdorff spaces);
\item whose {\it fibers} $\cE_x:=\pi^{-1}(x)$, $x\in X$ are Banach (respectively Hilbert) spaces;

\item so that scaling, addition and the norm are continuous;
 
\item and such that
  \begin{equation*}
    V_{U,\varepsilon}:=\{z\in \cE\ |\ \pi(z)\in U\text{ and }\|z\|<\varepsilon\}
    ,\quad
    \varepsilon>0
    ,\quad
    \text{nbhd }U\ni x
  \end{equation*}
  form a system of neighborhoods around the origin of the Banach space $\cE_x$ for every $x\in X$. 
\end{itemize}

The following statement highlights only a few items in the main result of the paper (\Cref{th:gogicgen}), giving various alternative characterizations of compact-base Banach bundles whose section module is topologically finitely-generated. In reference to item \Cref{item:thn:gogicgen:alllp}, recall (\cite[Example 2.6]{ryan_ban}, \cite[\S 6]{dfs_tens-3}) the sequence spaces
\begin{equation}\label{eq:lpe}
  \ell^p(E)
  :=
  \left\{(v_n)_{n\in \bZ_{\ge 0}}\in E^{\aleph_0}\ \bigg|\ \|(v_n)\|_p<\infty\right\}
  ,\quad
  1\le p\le \infty
\end{equation}
with
\begin{equation*}
  \|(v_{n})\|_p
  :=
  \begin{cases}
    \left(\sum_{n}\|v_{n}\|^p\right)^{1/p} &\text{ for }1\le p<\infty\\
    \sup_{n}\|v_{n}\| &\text{ for }p=\infty
  \end{cases}
\end{equation*}
attached to a Banach space $E$. 

\begin{theoremN}\label{thn:gogicgen}
  For a Banach bundle $\cE\xrightarrowdbl{}X$ over a compact Hausdorff space the following conditions are equivalent.
  \begin{enumerate}[(a),wide]

  \item\label{item:thn:gogicgen:subhom} There is a finite upper bound on the fiber dimensions $\dim \cE_x$, $x\in X$, the restrictions
    \begin{equation*}
      \cE|_{X_d}
      ,\quad
      X_{d} = X_{\cE=d}:=\{x\in X\ |\ \dim\cE_x=d\} \subseteq X
    \end{equation*}
    are trivialized by finite open covers, and the (automatically open) sets
    \begin{equation}\label{eq:xgd:intro}
      X_{>d} = X_{\cE>d}:=\{x\in X\ |\ \dim\cE_x>d\} \subseteq X
    \end{equation}
    are $F_{\sigma}$ (i.e. \cite[Problem 3H]{wil_top} countable unions of closed subsets of $X$).
    
  \item\label{item:thn:gogicgen:topfg} The Banach $C(X)$-module $E:=\Gamma(\cE)$ is topologically finitely generated. 
    
  \item\label{item:thn:gogicgen:alllp} \Cref{eq:xgd:intro} are $F_{\sigma}$ and for every (or equivalently, some) $1\le p<\infty$ there is a $K=K_p\in \bZ_{\ge 0}$ such that every element of $\ell^p(E)$ is decomposable as
    \begin{equation*}
      \ell^p(E)
      \ni
      (v_n)_n
      =
      \sum_{i=1}^K (f_{ni})\cdot w_i
      ,\quad
      (f_{ni})_n\in \ell^p(A)
      ,\quad
      w_i\in E. 
    \end{equation*}
    \qedhere
  \end{enumerate}
\end{theoremN}

Borrowing the term from the sheaf (\cite[Definition 8.3.19]{ks_shv-mfld}, \cite[pre Definition 4.5.3]{htt_d}) or topological-stability \cite[Definition 8.1]{zbMATH06124312} literature, the partition
\begin{equation*}
  X=\coprod_{d} X_d
\end{equation*}
of \Cref{thn:gogicgen}\Cref{item:thn:gogicgen:subhom} might be referred to as a {\it stratification} of $X$ (though note that in those other contexts the assumptions on the {\it strata} $X_d$ are stronger: they are often assumed smooth or topological manifolds, smooth analytic/algebraic varieties, etc.); this explains the phrasing of the abstract above. Given the local triviality \cite[Chapter II, Exercise 37]{fd_bdl-1} of the restrictions $\cE|_{X_d}$ to the locally compact Hausdorff strata $X_d$, and once more taking a cue from sheaf-theoretic terminology (\cite[Definition 4.5.3]{htt_d}, \cite[\S 2.2.10]{bbd}, etc.), one might refer to $\cE$ as {\it constructible} with respect to the stratification $(X_d)$. 

As for the equivalence \Cref{item:thn:gogicgen:topfg} $\Leftrightarrow$ \Cref{item:thn:gogicgen:alllp}, it ties in both to \cite[Theorem 3.6, (i) $\Leftrightarrow$ (ii)]{gog_top-fg} for $p=1$, and, for $1<p<\infty$, to Grothendieck's characterization \cite[paragraph following Proposition 5, p.90 and Corollaire 2, p.103]{groth_clases-suites_repr} of finite-dimensional Banach spaces $E$ as those for which the canonical map
\begin{equation*}
  \begin{aligned}
    \ell^p&\wo E
    \xrightarrow{\quad}
    \ell^p(E)
    ,\quad
    \ell^p:=\ell^p(\bC)\\
    &\wo:=\text{the {\it projective tensor product (\cite[\S I.1.1, D\'efinition 2]{groth_pt_1963}, \cite[\S 2.1]{ryan_ban}, etc.)}}
  \end{aligned}  
\end{equation*}
is onto; see \Cref{res:notfs}\Cref{item:res:notfs:grothpar} below for details. 

\subsection*{Acknowledgements}

I am grateful for valuable input from I. Gogi{\'c}. This work is partially supported by NSF grant DMS-2001128.


\section{Smallness conditions for Banach bundles}\label{se:smlbdl}

The discussion on \cite[p.8]{dg_ban-bdl} terms the Banach bundles recalled above {\it (F) Banach bundles} for Fell \cite[\S 1]{fell_ext}, contrasting them with {\it (H)} Banach bundles (for Hofmann, in work \cite{zbMATH03539905,hk_shv-bdl} relating such objects to {\it sheaves} of Banach spaces): as above, except that norms are only required to be {\it upper semicontinuous} (i.e. \cite[Problem 7K]{wil_top} preimages of intervals $(-\infty,a)\subset \bR$ are open). We employ the same terminology here, occasionally using {\it (F)} and {\it (H)} as adjectives, perhaps supplemented by other phrases for variety (e.g. {\it continuous} for {\it (F)}). 

A {\it section} of a bundle $\cE\xrightarrow{\pi}X$ is a continuous right inverse $X\xrightarrow{\sigma}\cE$ to $\pi$. The possibly-discontinuous right inverses to $\pi$, on the other hand, are sometimes \cite[\S 1, p.8]{dg_ban-bdl} termed {\it selections}. 

The following simple observation, of use in the subsequent discussion, imposes constraints on the loci where fiber dimensions of a Banach bundle are large. 

\begin{lemma}\label{le:bigdimfsigma}
  Let $\cE\xrightarrow{\pi}X$ be an (H) Banach bundle and $X\xrightarrow{\sigma_i}\cE$ countably many continuous sections thereof. The subspace
  \begin{equation}\label{eq:xsigmai}
    X_{\{\sigma_i\}\ge r} := \left\{x\in X\ |\ \dim\spn\left\{\sigma_i(x)\right\}\ge r\right\} \subseteq X
    ,\quad
    r\in \bZ_{\ge 0}
    \end{equation}
    is
    \begin{itemize}[wide]
    \item $F_{\sigma}$ (i.e. \cite[Problem 3H]{wil_top} a countable union of closed subsets of $X$);

    \item and open $F_{\sigma}$ if $\cE$ is (F). 
    \end{itemize}
  \end{lemma}
\begin{proof}
  The openness in the (F) case is noted in \cite[p.231, Remarque preceding \S 2]{dd}, and the $F_{\sigma}$ claim requires little more work (and none of any substance): write
  \begin{equation}\label{eq:infnormsemicont}
    \begin{aligned}
      X_{\{\sigma_i\}\ge r}
      &=
        \bigcup_{F\subseteq \{\sigma_i\},\ |F|=r} X_{F\ge r}
        \quad(\text{countable union}),\\
      X_{F\ge r}
      &:=
        \left\{x\in X\ \bigg|\ \inf_{(c_{\sigma})\in \bS^{2|F|-1}\subset \bC^{|F|}} \left\|\sum_{\sigma\in F}c_{\sigma}\sigma\right\|>0\right\}.
    \end{aligned}    
  \end{equation}
  The norm being upper semicontinuous by assumption, so is the infimum in \Cref{eq:infnormsemicont} and hence the preimage $X_{F\ge r}$ of $\bR_{>0}=\bigcup_{n\in \bZ>0}\left[\frac 1n,0\right)$ through that infimum is $F_{\sigma}$. 
\end{proof}

\begin{remarks}\label{res:skyscr}
  \begin{enumerate}[(1),wide]
  \item As the statement of \Cref{le:bigdimfsigma} suggests, there is no reason why \Cref{eq:xsigmai} would be open in general, for (H) bundles. For any compact Hausdorff $X$ and point $x$, for instance, there is an (H) Banach bundle $\cE$ on $X$ with 1-dimensional fiber $\cE_x$ and trivial fibers elsewhere: this is the bundle associated via \cite[Proposition 3.9]{zbMATH03539905} to the {\it skyscraper sheaf} (\spr{009A}, \cite[Exercise II.1.17]{hrt}) of Banach spaces on $X$ concentrated at $x$. For any section $\sigma$ that does not vanish at $x$ the locus $X_{\sigma\ge 1}=\{x\}$ is closed (so $F_{\sigma}$) but certainly not open.

  \item \cite[Example 3.14]{gog_top-fg}, of an (H) bundle whose fiber dimension jumps from 1 to 2 at a single exceptional point, is in the same spirit. In sheaf-theoretic language, it is the bundle associated through the same \cite[Proposition 3.9]{zbMATH03539905} to the {\it direct image} \cite[Definition 2.3.1]{ks_shv-mfld}
    \begin{equation*}
      \pi_*\cC([0,1])
      ,\quad
      \cC(\bullet):=\text{sheaf of continuous functions on }\bullet
    \end{equation*}
    through the quotient $[0,1]\xrightarrow{\pi}\bS^1$ identifying the two endpoints. 
  \end{enumerate}
\end{remarks}

For Banach spaces $E$ and $F$ we denote by $\wo$ the {\it projective (Banach) tensor product} of \cite[Definition A.3.67]{dales_autocont} (also \cite[\S 2.1]{ryan_ban}, \cite[\S 41.2]{k_tvs-2} and so on, with varying notation). If furthermore $E$ and $F$ are right and left {\it Banach modules} \cite[Definition 2.6.1]{dales_autocont} respectively over a Banach algebra $A$, $E\wo_AF$ denotes their projective tensor product \cite[pre Theorem II.3.5]{hlmsk_homolog} as such: 
\begin{equation*}
  E\wo_AF
  :=
  \text{Banach quotient }
  E\wo F/
  \overline{\spn}\left\{ea\otimes f-e\otimes af\ |\ a\in A,\ e\in E,\ f\in F\right\}. 
\end{equation*}

A Banach bundle $\cE\to X$ over a compact Hausdorff base $X$ provides a unital Banach $C(X)$-module $\Gamma(\cE)$ of sections thereof. Per \cite[Theorem 2.6]{dg_ban-bdl},
\begin{equation*}
  \bigg(\text{(F) Banach bundles on $X$}\bigg)
  \ni
  \cE
  \xmapsto{\qquad}
  \Gamma(\cE)
  \in
  \tensor[_{C(X)}]{\cM}{_{\cat{c}}}
  \quad
  \left(\text{`$\cat{c}$' for `convex'}\right)
\end{equation*}
is an equivalence between the category of (H) Banach bundles on $X$ and that of {\it convex} (unital Banach) $C(X)$-modules $M$ in the sense of \cite[discussion following Theorem 2.4]{dg_ban-bdl}:
\begin{equation*}
  \|fs+gt\|\le \max\left(\|s\|,\ \|t\|\right)
  ,\quad
  \forall\; 0\le f,g\in C(X)
  ,\quad
  f+g=1
  ,\quad
  s,t\in M.
\end{equation*}
We will henceforth switch perspective freely between bundles and modules. 

\Cref{def:projrk} is a minor gloss on \cite[Definition 3.3]{gog_top-fg}. Recall that a (left, say) Banach $A$-module $E$ is {\it non-degenerate} ({\it essential} in \cite[Definition 2.6.1]{dales_autocont}) if
\begin{equation*}
  \overline{AE}=E
  ,\quad
  AE:=\spn\left\{ae\ |\ a\in A,\ e\in E\right\}.
\end{equation*}
When $A$ is unital non-degeneracy is equivalent \cite[post Definition 2.6.1]{dales_autocont} to $E$ itself being unital, in the usual sense that $1e=e$ for all $e\in E$), which is the case we are mostly concerned with. 

\begin{definition}\label{def:projrk}
  Let $E$ and $F$ be non-degenerate left and right Banach $A$-modules.
  \begin{enumerate}[(1)]
  \item\label{item:def:projrk:el} The {\it ($A$-)projective rank} $\wh{\rk}(m)$ of an element $m\in F\wo_A E$ is the smallest $k$ (possibly $\infty$) for which
    \begin{equation*}
      m=\text{image of }\sum_{i=1}^k f_i\otimes e_i\in\text{ algebraic tensor product }F\otimes_A E.
    \end{equation*}
    
  \item\label{item:def:projrk:relmod} The {\it $F$-relative ($A$-)projective rank} of $E$ (possibly $\infty$) is
    \begin{equation*}
      \widehat{\rk}_F(E):=
      \inf\left\{k\in \bZ_{\ge 0}\ |\ \widehat{\rk}(m)\le k,\ \forall m\in F\wo_A E\right\}. 
    \end{equation*}

  \item\label{item:def:projrk:globmod} The {\it ($A$-)projective rank} of $E$ (possibly $\infty$) is
    \begin{equation*}
      \widehat{\rk}(E):=
      \sup_{\text{right non-degenerate }F}\widehat{\rk}_F(E). 
    \end{equation*}

  \item $E$ is {\it ($F$-relatively) of finite ($A$-)projective rank} if $\widehat{\rk}(E)<\infty$ (respectively $\widehat{\rk}_F(E)<\infty$).

  \item The mirror images of the definitions apply: one defines $E$-relative ($A$-)projective ranks for {\it right} modules $F$ in the obvious fashion, etc.
  \end{enumerate}
  For extra clarity, we might additionally decorate ranks with the algebra in question as a left-hand subscript, as in $\tensor[_A]{\wh{\rk}}{}(E)$. 
\end{definition}

Given Banach spaces $\left\{E_{\nu}\right\}_{\nu}$ and $1\le p\le \infty$ we write
\begin{equation*}
  \begin{aligned}
    \ell^p(\left\{E_{\nu}\right\})
    &:=
      \left\{(v_{\nu})_{\nu}\ \bigg|\ v_{\nu}\in E_{\nu}\quad\text{and}\quad\|(v_\nu)\|_p<\infty\right\}\\
    \|(v_\nu)\|_p
    &:=
      \begin{cases}
        \left(\sum_{\nu}\|v_{\nu}\|^p\right)^{1/p} &\text{ for }1\le p<\infty\\
        \sup_{\nu}\|v_{\nu}\| &\text{ for }p=\infty
      \end{cases}
  \end{aligned}  
\end{equation*}
for the {\it $\ell^p$-sum} \cite[\S 1.1, post Proposition 7]{hlmsk_fa} of the $E_{\nu}$. For countably infinite families of Banach spaces that happen to coincide with a single $E$ the corresponding $\ell^p$ sum specializes back to the $\ell^p(E)$ of \Cref{eq:lpe}, with plain $\ell^p$ short-hand for $\ell^p(\bC)$. Recall also (\cite[Example 2.6]{ryan_ban}, \cite[Proposition 5, p.90]{groth_clases-suites}) that $\ell^1(E) \cong \ell^1\wo E$ (for any Banach space $E$), so that 
\begin{equation}\label{eq:l1etens}
  \ell^1(E)
  \cong
  \ell^1\wo E
  \cong
  \left(\ell_1\wo A\right)\wo_{A}E
  \cong \ell^1(A)\wo_{A} E.
\end{equation}

\begin{remark}\label{re:notlp}
  In reference to \Cref{eq:l1etens}, recall (\cite[Proposition 4 and paragraph following Proposition 5]{groth_clases-suites_repr}; also \cite[Proposition 6.6]{dfs_tens-3}) that more generally, for $1\le p\le \infty$, the canonical map
  \begin{equation*}
    \ell^p\wo E
    \ni
    (x_n)_n\otimes\ v
    \xmapsto{\quad}
    (x_nv)_n
    \in
    \ell^p(E)
  \end{equation*}
  is a norm-1 embedding. For $1<p\le \infty$ that embedding is onto {\it precisely} when $E$ is finite-dimensional (\cite[Corollaire 2, p.103]{groth_clases-suites_repr} or \cite[Corollary 9.6]{dfs_tens-3} for $1<p<\infty$ and the already-cited \cite[paragraph post Proposition 5]{groth_clases-suites_repr} for $p=\infty$).
\end{remark}

The identification \Cref{eq:l1etens}, for the Banach algebra $A:=C_0(X)$ of continuous complex functions vanishing at infinity on locally compact Hausdorff $X$, shows that the hypothesis of \cite[Proposition 3.4]{gog_top-fg} for a non-degenerate Banach $C_0(X)$-module $E$ is nothing but the requirement that $E$ be $\ell^1(C_0(X))$-relatively of finite projective rank. This suggests a hierarchy of numerical invariants attached to a non-degenerate Banach $A$-module $E$:
\begin{equation}\label{eq:rks}
  \sup_{n\in \bZ_{\ge 0}} \widehat{\rk}_{\bC^n}(E)
  \quad\le\quad
  \widehat{\rk}_{\ell^1(A)}(E)
  \quad\le\quad
  \widehat{\rk}(E)
\end{equation}
(the first of which will turn up later: \Cref{re:walgfgcaution}). In this context, \cite[Proposition 3.4]{gog_top-fg} effectively proves that the last inequality is in fact an equality. That statement assumes $A=C_0(X)$ for locally compact Hausdorff $X$, but that assumption does not appear to be necessary. We recover that result in its full generality in \Cref{cor:allfin} below; the argument is essentially a rephrasing of the work just cited, presented, perhaps, so as to emphasize the substance of the matter. 

\Cref{re:notlp} also suggests a numerical invariant specific to $\ell^p(E)$ rather than tensor products.

\begin{definition}\label{def:lprk}
  Let $A$ be a Banach algebra, $E$ a non-degenerate left Banach $A$-module, and $1\le p\le\infty$.

  \begin{enumerate}[(1),wide]
  \item\label{item:def:lprk:el} The {\it ($A$-)$\ell^p$ rank} $\rk_p(m)$ of an element $m\in \ell^p(E)$ is the smallest $k$ (possibly $\infty$) for which
    \begin{equation*}
      m=\text{image of }\sum_{i=1}^k f_i\otimes e_i\in\text{ algebraic tensor product }\ell^p(A)\otimes_A E.
    \end{equation*}

  \item\label{item:def:lprk:relmod} The {\it ($A$-)$\ell^p$ rank} of $E$ (possibly $\infty$) is
    \begin{equation*}
      \rk_p(E):=
      \inf\left\{k\in \bZ_{\ge 0}\ |\ \rk_p(m)\le k,\ \forall m\in \ell^p(E)\right\}. 
    \end{equation*}
  \end{enumerate}
  As in \Cref{def:projrk}, the notation allows the specification of the algebra as a left subscript: $\tensor[_A]{\rk}{_p}(E)$, etc.
  
  By \Cref{eq:l1etens} we have $\rk_1(E)=\wh{\rk}_{\ell^1(A)}(E)$.
\end{definition}

\begin{remark}\label{re:rksrestrscalar}
  The various rank invariants in \Cref{def:projrk,def:lprk} play well with scalar restriction: if $B\le A$ is an embedding of Banach algebras and the non-degenerate $A$-module $E$ stays non-degenerate over $B$, then a straightforward unpacking of the definition shows that
  \begin{equation*}
    \tensor[_A]{rk}{_p}(E) \le \tensor[_B]{rk}{_p}(E);
  \end{equation*}
  the same goes for $\wh{\rk}$. 
\end{remark}

\begin{proposition}\label{pr:lprk}
  For a non-degenerate left Banach $A$-module $E$ we have
  \begin{equation*}
    \min_{1\le p\le \infty}\rk_p(E)
    =
    \rk_1(E)
    =
    \widehat{\rk}_{\ell^1(A)}(E)
    =
    \widehat{\rk}(E).
  \end{equation*}
\end{proposition}
\begin{proof}
  The middle equality is already noted in \Cref{def:lprk}, and
  \begin{equation*}
    \widehat{\rk}_{\ell^1(A)}(E)
    \quad\le\quad
    \widehat{\rk}(E)
    ,\quad
    \inf_{1\le p\le\infty}\rk_p(E)
  \end{equation*}
  in any case (\Cref{eq:rks} for the first inequality; the second is tautological). It will thus suffice to show that for an arbitrary $1\le p\le\infty$ and right Banach $A$-module $F$ we have $\widehat{\rk}_F(E) \le \rk_p(E)$; fully expanded, the claim is that 
  \begin{equation*}
    \bZ_{\ge 0}\ni K\ge \rk_p(E)
    \quad
    \xRightarrow{\quad}
    \quad
    K\ge
    \widehat{\rk}_{F}(E)
    ,\quad\forall
    \text{ non-degenerate right $A$-module $F$}.
  \end{equation*}
  The desired conclusion says that the obvious {\it pairing} map
  \begin{equation}\label{eq:pairf}
    \begin{tikzpicture}[>=stealth,auto,baseline=(current  bounding  box.center)]
      \path[anchor=base] 
      (0,0) node (l) {$\left(F\otimes V^*\right)\times \left(V\otimes E\right)$}
      +(3,1) node (ul) {$F\otimes V^*\otimes V\otimes E$}
      +(8,1) node (ur) {$F\otimes E$}
      +(9,0) node (r) {$F\wo_A E$}
      ;
      \draw[->] (l) to[bend left=6] node[pos=.5,auto] {$\scriptstyle $} (ul);
      \draw[->] (ul) to[bend left=6] node[pos=.5,auto] {$\scriptstyle \id_F\otimes\cat{eval}\otimes \id_E$} (ur);
      \draw[->] (ur) to[bend left=6] node[pos=.5,auto,swap] {$\scriptstyle $} (r);
      \draw[->] (l) to[bend right=6] node[pos=.5,auto,swap] {$\scriptstyle \cat{pair}_F$} (r);
    \end{tikzpicture}
  \end{equation}
  is onto for $V:=\bC^K$, where unadorned `$\otimes$' denotes plain algebraic tensor products and $V^*\otimes V\xrightarrow{\cat{eval}}\bC$ is the guessable bilinear map. To check this, fit \Cref{eq:pairf} as the top right-hand map in the commutative diagram
  \begin{equation*}
    \begin{tikzpicture}[>=stealth,auto,baseline=(current  bounding  box.center)]
      \path[anchor=base] 
      (0,0) node (l) {$\ell^{p'}(F)\times \left(\ell^p(A)\otimes V^*\right)\times V\otimes E$}
      +(6,.5) node (u) {$F\otimes V^*\times V\otimes E$}
      +(6,-.5) node (d) {$\ell^{p'}(F)\times \ell^p(E)$}
      +(9,0) node (r) {$F\wo_AE$}
      ;
      \draw[->] (l) to[bend left=6] node[pos=.5,auto] {$\scriptstyle \square\times\id_{V\otimes E}$} (u);
      \draw[->] (u) to[bend left=6] node[pos=.5,auto] {$\scriptstyle \cat{pair}_F$} (r);
      \draw[->>] (l) to[bend right=6] node[pos=.5,auto] {$\scriptstyle $} node[pos=.5,auto,swap] {$\scriptstyle \id_{\ell^{p'}(F)}\times\cat{pair}_{\ell^p(A)}$} (d);
      \draw[->>] (d) to[bend right=6] node[pos=.5,auto,swap] {$\scriptstyle \left((f_n),(e_n)\right)\mapsto \sum_n f_n\otimes e_n$} (r);
    \end{tikzpicture}
  \end{equation*}
  where
  \begin{itemize}[wide]
  \item $p'$ is the {\it conjugate exponent} \cite[Definition 3.4]{rud_rc_3e_1987} to $p$: $\frac 1{p}+\frac 1{p'}=1$ (so that $1'=\infty$);
  \item the lower left-hand arrow is onto by assumption;
  \item the lower right-hand is onto by \cite[Proposition 2.8]{ryan_ban};
  \item and $\square$ pairs elements of $\ell^{p'}(F)$ and $\ell^p(A)$ by
    \begin{equation*}
      \ell^{p'}(F)\times \ell^p(A)
      \ni
      \left((f_n),\ (a_n)\right)
      \xmapsto{\quad}
      \sum_n f_n a_n
      \in
      F.
    \end{equation*}
  \end{itemize}
  The bottom composition being onto so is the top, hence the sought-for surjectivity of the top right hand. 
\end{proof}


An immediate consequence of \Cref{pr:lprk}:

\begin{corollary}\label{cor:allfin}
  For a non-degenerate left Banach $A$-module $E$ the following conditions are equivalent.
  \begin{enumerate}[(a),wide]
  \item $E$ is of finite projective rank.

  \item $E$ is $\ell^1(A)$-relatively of finite projective rank.

  \item $E$ has finite $\ell^p$ rank for some $1\le p\le\infty$.  \qedhere
  \end{enumerate}
\end{corollary}

\begin{remark}\label{re:walgfgcaution}
  {\it Weak algebraic finite generation} \cite[p.560, item (a)]{gog_top-fg}, the other finiteness notion pertinent to \cite[Theorem 3.6]{gog_top-fg}, can also be phrased in terms of relative projective ranks. Some caution is in order, as that property seems to be slightly misstated in loc. cit. as well as later, in \cite[Remark 3.5]{gog_top-fg}:
  \begin{itemize}[wide]
  \item On the one hand, the definition demands that every finitely-generated (henceforth {\it f.g.}) module be generated by $\le K$ elements for some constant $K$. 

  \item On the other, this is  more than what \cite[Remark 3.5]{gog_top-fg} yields. The latter shows that for some $K$, every finite set $\{v\}$ of elements in the module is in the $C(X)$-linear span of at most $K$ elements. These latter elements will not be, in general, again in the module generated by the original set $\{v\}$.
  \end{itemize}
  The right property to ask for, then, is that for some $K$ every f.g. module be {\it contained} in a module generated by $\le K$ elements. Minding this caveat, it is {\it this} that we refer to as being {\it weakly (algebraically) f.g.} in the sequel: in the language of \Cref{def:projrk}\Cref{item:def:projrk:relmod}, the requirement that
  \begin{equation}\label{eq:realwk}
    \sup_{n\in \bZ_{\ge 0}} \widehat{\rk}_{\bC^n}(E)<\infty.    
  \end{equation}  
\end{remark}


Some of the discrepancy between the two ``competing'' notions of weak algebraic finite generation is already visible in the example discussed in \cite[\S 1, 2$^{nd}$ paragraph]{gog_top-fg}. 

\begin{example}\label{ex:weakfgwrong}
  The Hilbert (hence also Banach) module will be the ideal
  \begin{equation*}
    C_0((0,1])\triangleleft C([0,1])
  \end{equation*}
  over the latter $C^*$-algebra. Loc. cit. notes that the module is topologically {\it cyclic} (i.e. generated by a single element); for that reason, the supremum \Cref{eq:realwk} is 1. On the other hand, it is certainly {\it not} the case that every f.g. ideal $I\le C_0((0,1])$ of $C([0,1])$ is principal. This follows, for example, from \cite[Theorem 2.3]{gh_fg-2}, which ensures that if $f\in C_0((0,1])$ is any piecewise linear function that takes both positive and negative values then $(f,|f|)$ cannot be principal. 
\end{example}

Following \cite[\S 2]{gog_top-fg} with minor alterations, call an (F) Banach bundle over $X$
\begin{itemize}[wide]
\item {\it ($n$-)homogeneous} if $\dim\cE_x$ is constant (respectively equal to $n\in \bZ_{\ge 0}$) for all $x\in X$;
\item {\it ($n$-)subhomogeneous} if $\dim\cE_x$ is bounded (respectively by $n\in \bZ_{\ge 0}$) for all $x\in X$;
\item {\it of finite type (or f.t.)} if $X$ admits a finite open cover $X=\bigcup_i V_i$ with all $\cE|_{V_i}$ trivial;
\item and {\it conditionally of finite type (conditionally f.t.)} if the restriction of $\cE$ to any subset where its fibers are equidimensional is of finite type. 
\end{itemize}


\begin{theorem}\label{th:gogicgen}    
  Let $X$ be a compact Hausdorff space and $\cE$ an (F) Banach bundle over $X$. The following conditions are equivalent:
  \begin{enumerate}[(a),wide]
  \item\label{item:th:gogicgen:subhom} $\cE$ is subhomogeneous and conditionally f.t., and the (automatically open) sets
    \begin{equation}\label{eq:xgd}
      X_{>d} = X_{\cE>d}:=\{x\in X\ |\ \dim\cE_x>d\} \subseteq X
    \end{equation}
    are $F_{\sigma}$. 
    
  \item\label{item:th:gogicgen:topfg} The Banach $C(X)$-module $E:=\Gamma(\cE)$ is topologically finitely generated.

  \item\label{item:th:gogicgen:alllp} \Cref{eq:xgd} are $F_{\sigma}$ and $E$ has finite $\ell^p$ rank for all $1\le p<\infty$.
    
  \item\label{item:th:gogicgen:somelp} \Cref{eq:xgd} are $F_{\sigma}$ and $E$ has finite $\ell^p$ rank for some $1\le p<\infty$. 
    
  \item\label{item:th:gogicgen:l1proj} \Cref{eq:xgd} are $F_{\sigma}$ and $E$ is $\ell^1(C(X))$-relatively of finite $C(X)$-projective rank.

  \item\label{item:th:gogicgen:proj} \Cref{eq:xgd} are $F_{\sigma}$ and $E$ is of finite $C(X)$-projective rank.

  \item\label{item:th:gogicgen:walgfg} \Cref{eq:xgd} are $F_{\sigma}$ and $E$ is weakly algebraically f.g. (in the sense of \Cref{eq:realwk}). 
  \end{enumerate}
\end{theorem}

Before going into the proof, we record the following consequence; it is an immediate application of \Cref{th:gogicgen} to Banach bundles with 1-dimensional fiber over an open locus and vanishing over a closed subspace (so subhomogeneous of maximal fiber dimension $\le 1$).

\begin{corollary}\label{cor:zeroloci}
  Let $Z\subseteq Y$ be a closed embedding of compact Hausdorff spaces.

  For the ideal $C_0(Y\setminus Z) \trianglelefteq C(Y)$ regarded as a $C(Y)$-module the conditions \Cref{item:th:gogicgen:topfg} to \Cref{item:th:gogicgen:walgfg} of \Cref{th:gogicgen} are equivalent, and met precisely when $Z$ is $G_{\delta}$ in $Y$ (i.e. \cite[Problem 3H]{wil_top} a countable intersection of open subsets of $Y$).  \qedhere
\end{corollary}

It will be convenient to isolate part of \Cref{cor:zeroloci} (in a slightly sharper form) for later use in the proof of \Cref{th:gogicgen}.

\begin{lemma}\label{le:gogicgen-ideals}
  Let $Z\subseteq Y$ be a closed $G_{\delta}$ subspace of a compact Hausdorff space. Setting $U:=Y\setminus Z$, we have
  \begin{equation}\label{eq:le:gogicgen-ideals:scalarrestr}
    \tensor[_{C_0(U)}]{rk}{_p}(C_0(U)) \le 1
    \xRightarrow{\quad}
    \tensor[_{C(X)}]{rk}{_p}(C_0(U)) \le 1
    ,\quad
    \forall 1\le p<\infty. 
  \end{equation}
\end{lemma}
\begin{proof}
  Write
  \begin{equation}\label{eq:zint}
    Z=\bigcap_{m\ge 0}U_m
    ,\quad
    \ol{U_{m+1}}\subseteq U_m
    ,\quad
    U_m\subseteq X\text{ open}
  \end{equation}
  and fix $(f_n)_{n\ge 0}\in \ell^p(E)$ for $E:=C_0(U)$. Because for each fixed $n$ the sequence $\left(\|f_n\|_{\ol{U_m}}\right)_m$ descends to $0$ with $m$, the {\it Monotone Convergence Theorem} \cite[\S 21.38]{scht_hndbk} shows that
  \begin{equation}\label{eq:lm}
    \lambda_m:=
    \left\|\left(f_n|_{\ol{U_m}}\right)_n\right\|_p
    \xrightarrow[\quad m\to\infty\quad]{\quad}
    0
  \end{equation}
  Note, next, that for every $k\in \bZ_{\ge 0}$ there is some $m_k$ with
  \begin{equation}\label{eq:fn122k}
    \left\|f_n\right\|_{\ol{U_{m_k}}} < \frac 1{2^{2k}}
    ,\quad
    \forall n\in \bZ_{\ge 0}. 
  \end{equation}
  {\it Urysohn's Lemma} \cite[Theorem 33.1]{mnk} (or rather its proof) implies the existence of a continuous function $Y\xrightarrow{f}\bR_{\ge 0}$ vanishing {\it precisely} along $Z$ (this already uses the fact that $Z$ is $G_{\delta}$) and such that
  \begin{align*}
    f|_{\ol{U_m}}&>\lambda_m
                   ,\quad\forall m
                   \quad\text{and}\numberthis\label{eq:flmbd}
    \\
    f|_{\ol{U_{m_k}}} &> \frac 1{2^k}
                        ,\quad\forall k.\numberthis\label{eq:f12k}
  \end{align*}
  \Cref{eq:fn122k,eq:f12k} show that the functions $g_n:=\frac {f_n}{f}$ on $U$ vanish at $\infty$ thereon and hence extend by 0 across $Z$, while \Cref{eq:flmbd}, by the very definition of $\lambda_m$ in \Cref{eq:lm}, shows that $(g_n)_n\in \ell^p(C_0(U))$. We have now factored $(f_n)_n$ as
  \begin{equation*}
    (f_n) = (g_n)\cdot f
    ,\quad
    (g_n)\in \ell^p(C_0(U))
    ,\quad
    f\in C_0(U),
  \end{equation*}
  hence the first claimed inequality in \Cref{eq:le:gogicgen-ideals:scalarrestr}; the implication, on the other hand, follows from \Cref{re:rksrestrscalar}.
\end{proof}

\pf{th:gogicgen}
\begin{th:gogicgen}
  In part, the argument consists of running through the proofs of \cite[Theorems 1.1 and 3.6]{gog_top-fg} and making appropriate modifications to remove the metrizability assumptions therein. 
  
  \begin{enumerate}[label={},wide]

    
  \item {\bf \Cref{item:th:gogicgen:topfg} $\Rightarrow$ \Cref{item:th:gogicgen:subhom}:} The $F_{\sigma}$ claim follows from \Cref{le:bigdimfsigma} while \cite[Theorem 1.1, (i) $\Rightarrow$ (ii)]{gog_top-fg} (which makes no use of metrizability) settles the rest. 
    
  \item {\bf \Cref{item:th:gogicgen:subhom} $\Rightarrow$ \Cref{item:th:gogicgen:topfg}, \Cref{item:th:gogicgen:alllp}:} By simple direct verification, $\ell^p(\bullet)$ is an {\it exact} endofunctor on the category $\tensor[_A]{\cM}{}$ of non-degenerate Banach $A$-modules in the sense (stronger than what is sometimes \cite[Definition III.3.6]{hlmsk_homolog} meant by the phrase) that for every {\it short exact sequence} \cite[Definition 0.5.1]{hlmsk_homolog}
    \begin{equation*}
      0\to
      X
      \lhook\joinrel\xrightarrow[\quad\text{closed embedding}\quad]{\quad}
      Y
      \xrightarrowdbl{\quad}
      Y/X
      \to
      0
    \end{equation*}
    in $\tensor[_A]{\cM}{}$ the corresponding sequence obtained by substituting $\ell^p(\bullet)$ for $\bullet$ is again exact. For that reason, the desired finite-$\ell^p$-rank property is inherited by the middle term $Y$ of such a sequence (or {\it extension}) from $X$ and $Y/X$; in short: the property of having finite $\ell^p$ rank is preserved by extensions. So too is topological finite generation, so we can treat the two in parallel. 

    The global-section Banach $C(X)$-module $E=\Gamma(\cE)$ can be obtained as an extension:
    \begin{equation}\label{eq:extsect}
      0\to
      \Gamma_0(\cE|_{X_{d}})
      \lhook\joinrel\xrightarrow[]{\quad}
      \Gamma(\cE)
      \xrightarrowdbl{\quad}
      \Gamma(\cE|_{X_{<d}})
      \to
      0,
    \end{equation}
    where
    \begin{itemize}[wide]
    \item $d$ is the maximal fiber dimension of $\cE$;
    \item we write
      \begin{equation*}
        \begin{aligned}
          X_{d}
          &= X_{\cE=d}:=\{x\in X\ |\ \dim\cE_x=d\} \subseteq X\\
          X_{<d}
          &= X_{\cE<d}:=\{x\in X\ |\ \dim\cE_x<d\}
            = X\setminus X_d
        \end{aligned}        
      \end{equation*}
      (consistently with \Cref{eq:xgd});

    \item and $\Gamma_0(-)$ denotes sections vanishing at $\infty$ on the {\it locally} compact Hausdorff space $X_d$ (open in $X$). 
    \end{itemize}
    We proceed inductively on maximal fiber dimension: assuming the conclusion holds for the rightmost term $\Gamma(\cE|_{X_{<d}})$ of \Cref{eq:extsect}, the noted finite-$\ell^p$-rank permanence under extensions reduces the problem to showing that the leftmost term $\Gamma_0(\cE|_{X_{d}})$ has the requisite property.
    
    We are also assuming $X_d$ admits a finite open cover by $U_i$, $1\le i\le n$ trivializing $\cE$: 
    \begin{equation*}
      \cE|_{U_i}\cong\text{trivial bundle with fiber }\bC^d
      ,\quad
      1\le i\le n.
    \end{equation*}
    By \cite[Theorem 20.12(c)]{wil_top} the $F_{\sigma}$ subspace $X_d\subseteq X$ of the compact Hausdorff $X$ is {\it paracompact} \cite[Definition 20.6]{wil_top}, hence \cite[Theorem 41.7]{mnk} continuous functions
    \begin{equation*}
      X_d\xrightarrow{\quad\phi_i\quad}[0,1]
      ,\quad
      \mathrm{supp}~\phi_i
      :=
      \ol{\phi_i^{-1}(\bR_{>0})}
      \subset U_i
      ,\quad
      \sum_{i=1}^n\phi_i(x)=1,\ \forall x\in X
    \end{equation*}
    (i.e. a {\it partition of unity} \cite[Definition preceding Lemma 41.6]{mnk} subordinate to the cover $(U_i)_i$). There are $C(X)$-module maps
    \begin{equation*}
      \begin{tikzpicture}[>=stealth,auto,baseline=(current  bounding  box.center)]
        \path[anchor=base] 
        (0,0) node (l) {$\bigoplus_{i=1}^n\Gamma_0(\cE|_{U_i})$}
        +(6,0) node (r) {$\Gamma_0(\cE|_{X_d})$}
        ;
        \draw[->] (l)
        to[bend left=20] node[pos=.5,auto]
        {$\scriptstyle
          (\sigma_i)_{i=1}^n
          \longmapsto
          \sum_{i=1}^n \sigma_i$}
        (r);
        \draw[->] (r)
        to[bend left=20] node[pos=.5,auto]
        {$\scriptstyle
          (\phi_i\sigma)_i
          \longmapsfrom
          \sigma$}
        (l);
      \end{tikzpicture}
    \end{equation*}
    realizing the right-hand side as a summand of the left. It is thus enough to handle the individual summands
    \begin{equation*}
      \Gamma_0(\cE|_{U_i})
      \cong
      C_0(U_i)^d
      \quad
      \big(\text{by the triviality of the restrictions }\cE|_{U_i}\big). 
    \end{equation*}
    The problem, then, has been reduced to ideals (rather than broader classes of Banach modules): the issue is to deduce \Cref{item:th:gogicgen:topfg} topological finite generation and \Cref{item:th:gogicgen:alllp} the finite-$\ell^p$-rank property for the ideal
    \begin{equation}\label{eq:c0yi}
      C_0(U_i)\triangleleft C\left(\overline{U_i}\right)
      \quad
      \left(\text{$\ol{\bullet}$ meaning closure in $X$}\right),
    \end{equation}
    assuming that $U_i\subseteq \overline{U_i}$ is open $F_{\sigma}$. \Cref{le:gogicgen-ideals} with $Y:=\overline{U_i}$ and $U:=U_i$ handles \Cref{item:th:gogicgen:alllp}. As for \Cref{item:th:gogicgen:topfg}, note that there are continuous functions on $\overline{U_i}$ vanishing {\it precisely} on the closed $G_{\delta}$ subset $\overline{U_i}\setminus U_i$ \cite[\S 33, Exercise 4]{mnk}, and any such function will generate \Cref{eq:c0yi} topologically: the latter claim is an immediate consequence of the classification \cite[Problem 4O]{gj_rings} of the (norm-)closed ideals of $C(Y)$ as precisely
    \begin{equation}\label{eq:0ideal}
      C_{0;Z}(Y):=\{f\in C(Y)\ |\ f|_{Z}\equiv 0\}
      ,\quad
      Z\subseteq Y\text{ closed}.
    \end{equation}
    
    
  \item {\bf \Cref{item:th:gogicgen:alllp} $\Rightarrow$ \Cref{item:th:gogicgen:somelp}} is formal. 
    
  \item {\bf \Cref{item:th:gogicgen:somelp} $\Leftrightarrow$ \Cref{item:th:gogicgen:l1proj} $\Leftrightarrow$ \Cref{item:th:gogicgen:proj}:} \Cref{cor:allfin}.   
    
    
  \item {\bf \Cref{item:th:gogicgen:l1proj} $\Rightarrow$ \Cref{item:th:gogicgen:walgfg}:} This follows from the first inequality in \Cref{eq:rks}: the hypothesis is that the middle term in that inequality is finite, hence so is the leftmost term. 
    
  \item {\bf \Cref{item:th:gogicgen:walgfg} $\Rightarrow$ \Cref{item:th:gogicgen:topfg}:} Once more an implication that transports over directly: \cite[Theorem 3.6, (iv) $\Rightarrow$ (i)]{gog_top-fg} appeals to \cite[Lemma 2.11]{gog_subhom-1}, which in turn requires that the sets $X_{>d}$ be $\sigma$-compact; they are here, by the $F_{\sigma}$ assumption. 
  \end{enumerate}
  This renders the directed implication graph strongly connected.
\end{th:gogicgen}

\begin{remarks}\label{res:notfs}
  \begin{enumerate}[(1),wide]
  \item\label{item:res:notfs:1st2} It is perfectly possible for the first two conditions of \Cref{th:gogicgen}\Cref{item:th:gogicgen:subhom} to hold (for an (F) bundle over compact $X$) without $X_{>d}$ being $F_{\sigma}$. In fact, there is no a priori constraint on those sets at all, beyond being open: let $Y\subseteq X$ be an arbitrary embedding of compact Hausdorff spaces, and set $\cE$ to be the Hilbert bundle over $X$ attached to the Hilbert $C(X)$-module (and ideal) $C_{0;Y}(X)$ in the notation of \Cref{eq:0ideal}. We then have $X_{>0}=X\setminus Y$, an arbitrary open subset of an arbitrary compact Hausdorff space.

    What is more, \Cref{ex:omega} below shows that the $F_{\sigma}$ condition is not automatic in any of the items in \Cref{th:gogicgen} assuming it: \Cref{item:th:gogicgen:subhom} and \Cref{item:th:gogicgen:alllp} to \Cref{item:th:gogicgen:walgfg}. 

  \item\label{item:res:notfs:altpf} The proof of \Cref{th:gogicgen} deduces \Cref{item:th:gogicgen:topfg} and \Cref{item:th:gogicgen:alllp} from \Cref{item:th:gogicgen:subhom} uniformly, using the fact that both properties survive extensions. Alternatively, the implication \Cref{item:th:gogicgen:subhom} $\Rightarrow$ \Cref{item:th:gogicgen:topfg} on its own could have been recovered from the proof of \cite[Theorem 1.1, (ii) $\Rightarrow$ (i)]{gog_top-fg} via \cite[Proposition 3.2, (i) $\Rightarrow$ (ii)]{gog_top-fg}. Metrizability manifests there only in ensuring that the sets $X_{>d}$ (indeed automatically open by \cite[Proposition 1.6]{dupre_hilbund-1}) are {\it $\sigma$-compact} (i.e. \cite[Problem 17I]{wil_top} countable unions of compact subsets). Since they are here assumed $F_{\sigma}$ in the compact Hausdorff $X$, they still are $\sigma$-compact in the context of the proof above.  

  \item\label{item:res:notfs:grothpar} Note the parallels between the equivalence \Cref{item:th:gogicgen:topfg} $\Leftrightarrow$ \Cref{item:th:gogicgen:alllp} of \Cref{th:gogicgen} and the aforementioned (\Cref{re:notlp}) properness of the embedding $\ell^p\wo E\lhook\joinrel\to \ell^p(E)$ for $1<p\le p$ and infinite-dimensional $E$: condition \Cref{item:th:gogicgen:alllp} of \Cref{th:gogicgen} implies the surjectivity of the map
    \begin{equation*}
      \ell^p(A)\otimes_A E
      \xrightarrow{\quad}
      \ell^p(E)
    \end{equation*}
    defined on the {\it algebraic} tensor product (a {\it stronger} assumption than an $A$-relative analogue of $\ell^p\wo E=\ell^p(E)$), and implies the finiteness of \Cref{th:gogicgen}\Cref{item:th:gogicgen:topfg}. 
    
  \item\label{item:res:notfs:noinfty} There are good reasons for excluding the case $p=\infty$ from both \Cref{th:gogicgen}\Cref{item:th:gogicgen:alllp} and \Cref{le:gogicgen-ideals}: not only do those claims not hold, but in fact a strong negation thereof does (\Cref{le:isclopen}). Compare the latter to \cite[Corollary 2.4]{gh_fg-2}, very much in the same spirit: given a {\it Tychonoff space} \cite[Definition 14.8]{wil_top} $X$ with the property that finitely-generated ideals in the ring of continuous real-valued functions on $X$ are principal, points $x\in X$ with countable neighborhood bases are by necessity isolated. 
  \end{enumerate}  
\end{remarks}

\begin{lemma}\label{le:isclopen}
  For an open $F_{\sigma}$ subset $U\subseteq Y$ of a compact Hausdorff space the following conditions are equivalent.

  \begin{enumerate}[(a),wide]
  \item\label{item:le:isclopen:clp} $U$ is clopen (= closed and open) in $X$.

  \item\label{item:le:isclopen:rk1} The $\ell^{\infty}$ rank $\tensor[_{C(Y)}]{\rk}{_{\infty}}(C_0(U))$ is 1.

  \item\label{item:le:isclopen:rkfin} $\tensor[_{C(Y)}]{\rk}{_{\infty}}(C_0(U))<\infty$. 
  \end{enumerate}
\end{lemma}
\begin{proof}
  The interesting implication is \Cref{item:le:isclopen:rkfin} $\Rightarrow$ \Cref{item:le:isclopen:clp}. Exhibit $Z:=Y\setminus U$ once more as an intersection \Cref{eq:zint} and consider (Urysohn \cite[Theorem 33.1]{mnk} again) continuous functions
  \begin{equation}\label{eq:fumfum}
    Y\xrightarrow{\quad f_m\quad}[0,1]
    ,\quad
    f_m|_{\ol{U_{m+1}}}\equiv 0
    ,\quad
    f_m|_{Y\setminus U_m}\equiv 1
    ,\quad
    \forall m\in \bZ_{\ge 0}. 
  \end{equation}
  Condition \Cref{item:le:isclopen:rkfin} means that we can decompose the element $(f_m)\in \ell^{\infty}(C_0(U))$ as a finite linear combination with $C_0(U)$-coefficients of elements in $\ell^{\infty}(C(Y))$:
  \begin{equation*}
    f_m = \sum_{i=1}^s g_{mi}\cdot \varphi_i
    ,\quad
    (g_{mi})_m\in \ell^{\infty}(C(Y))
    ,\quad
    \varphi_i\in C_0(U). 
  \end{equation*}
  There is a uniform upper bound on $\|g_{mi}\|$ (i.e. independent of $m$ and $i$). Coupled with the fact that $\varphi_i$, $1\le i\le s$ are all small on $U_n$ for $n$ sufficiently large, this shows that so too are $f_m$. But this then implies (in conjunction with \Cref{eq:fumfum}) that the $U_m$ stabilize for large $m$: clopenness, in other words. 
\end{proof}

As indicated in \Cref{res:notfs}\Cref{item:res:notfs:1st2}, the following simple example shows that the $F_{\sigma}$ assumption cannot be dropped in any of the conditions of \Cref{th:gogicgen} which assume it. 

\begin{example}\label{ex:omega}
  Take for $X$ the compact space $[0,\Omega]$ of \cite[\S II.43]{ss_countertop}: all ordinals less than or equal to the first uncountable one, $\Omega$, with the {\it order topology} (\cite[\S II.39]{ss_countertop}, \cite[\S 14]{mnk}, etc.). The Hilbert bundle will be that attached to the Hilbert $C(X)$-module $C_0(U)$, $U:=[0,\infty)$; it is trivial of rank 1 over $U$ and vanishes at $\Omega$. In particular, $X_{>0}=U$ is {\it not} $F_{\sigma}$ \cite[\S II.39, item 10]{ss_countertop}. 

  Every element of $C_0(U)$ in fact vanishes on an entire neighborhood of $\Omega$ \cite[\S II.43, item 12]{ss_countertop}, so every element of $\ell^p(C_0(U))$ also belongs to some $\ell^p(C_0([0,\alpha)))$ for some (countable) ordinal $\alpha>\omega$. An application of \Cref{le:gogicgen-ideals} to the space $G_{\delta}$ subspace $\{\alpha\}\subset [0,\alpha]$ then shows that the $\ell^p$ ranks are all $1$, hence the second condition in \Cref{th:gogicgen}\Cref{item:th:gogicgen:alllp}. Analogous arguments apply for the other conditions involving the $F_{\sigma}$ requirement. 
\end{example}


\addcontentsline{toc}{section}{References}

\begin{thebibliography}{10}

\bibitem{stacks-project}
The Stacks~Project Authors.
\newblock Stacks project.

\bibitem{bbd}
A.~A. Beilinson, J.~Bernstein, and Pierre Deligne.
\newblock {\em Analyse et topologie sur les espaces singuliers {(I)}}.
\newblock Number 100 in Ast\'erisque. Soci\'et\'e math\'ematique de France,
  1982.

\bibitem{2402.08121v1}
Alexandru Chirvasitu.
\newblock Orbit misbehavior, isotropy discontinuity, and large isotypic
  components, 2024.
\newblock \url{http://arxiv.org/abs/2402.08121v1}.

\bibitem{dales_autocont}
H.~G. Dales.
\newblock {\em Banach algebras and automatic continuity}, volume~24 of {\em
  London Mathematical Society Monographs. New Series}.
\newblock The Clarendon Press, Oxford University Press, New York, 2000.
\newblock Oxford Science Publications.

\bibitem{dfs_tens-3}
Joe Diestel, Jan Fourie, and Johan Swart.
\newblock The metric theory of tensor products ({G}rothendieck's r\'{e}sum\'{e}
  revisited). {III}. {V}ector sequence spaces.
\newblock {\em Quaest. Math.}, 25(1):95--118, 2002.

\bibitem{dd}
Jacques Dixmier and Adrien Douady.
\newblock Champs continus d'espaces hilbertiens et de {$C^{\ast} $}-alg\`ebres.
\newblock {\em Bull. Soc. Math. France}, 91:227--284, 1963.

\bibitem{dg_ban-bdl}
M.~J. Dupr{\'e} and R.~M. Gillette.
\newblock {\em Banach bundles, {Banach} modules and automorphisms of
  {{\(C^*\)}}-algebras}, volume~92 of {\em Res. Notes Math., San Franc.}
\newblock Pitman Publishing, London, 1983.

\bibitem{dupre_hilbund-1}
Maurice~J. Dupre.
\newblock Classifying {Hilbert} bundles.
\newblock {\em J. Funct. Anal.}, 15:244--278, 1974.

\bibitem{fell_ext}
J.~M.~G. Fell.
\newblock {\em An extension of {Mackey}'s method to {Banach} {{\(^
  *\)}}-algebraic bundles}, volume~90 of {\em Mem. Am. Math. Soc.}
\newblock Providence, RI: American Mathematical Society (AMS), 1969.

\bibitem{fd_bdl-1}
J.~M.~G. Fell and R.~S. Doran.
\newblock {\em Representations of *-algebras, locally compact groups, and
  {Banach} *- algebraic bundles. {Vol}. 1: {Basic} representation theory of
  groups and algebras}, volume 125 of {\em Pure Appl. Math., Academic Press}.
\newblock Boston, MA etc.: Academic Press, Inc., 1988.

\bibitem{gh_fg-2}
Leonard Gillman and Melvin Henriksen.
\newblock Rings of continuous functions in which every finitely generated ideal
  is principal.
\newblock {\em Trans. Am. Math. Soc.}, 82:366--391, 1956.

\bibitem{gj_rings}
Leonard Gillman and Meyer Jerison.
\newblock {\em Rings of continuous functions}.
\newblock Graduate Texts in Mathematics, No. 43. Springer-Verlag, New
  York-Heidelberg, 1976.
\newblock Reprint of the 1960 edition.

\bibitem{gog_subhom-1}
Ilja Gogi{\'c}.
\newblock Elementary operators and subhomogeneous {{\(C^*\)}}-algebras.
\newblock {\em Proc. Edinb. Math. Soc., II. Ser.}, 54(1):99--111, 2011.

\bibitem{gog_top-fg}
Ilja Gogi{\'c}.
\newblock Topologically finitely generated {Hilbert} {{\(C(X)\)}}-modules.
\newblock {\em J. Math. Anal. Appl.}, 395(2):559--568, 2012.

\bibitem{groth_clases-suites}
A.~Grothendieck.
\newblock Sur certaines classes de suites dans les espaces de {B}anach et le
  th\'{e}or\`eme de {D}voretzky-{R}ogers.
\newblock {\em Bol. Soc. Mat. S\~{a}o Paulo}, 8:81--110, 1953.

\bibitem{groth_pt_1963}
Alexandre Grothendieck.
\newblock {\em Produits tensoriels topologiques et espaces nucl{\'e}aires},
  volume~16 of {\em Mem. Am. Math. Soc.}
\newblock Providence, RI: American Mathematical Society (AMS), reprint edition,
  1963.

\bibitem{groth_clases-suites_repr}
Alexandre Grothendieck.
\newblock Sur certaines classes de suites dans les espaces de {B}anach, et le
  th\'{e}or\`eme de {D}voretzky-{R}ogers.
\newblock {\em Resenhas}, 3(4):447--477, 1998.
\newblock With a foreword by Paulo Cordaro, Reprint of the 1953 original.

\bibitem{hrt}
Robin Hartshorne.
\newblock {\em Algebraic geometry}.
\newblock Springer-Verlag, New York-Heidelberg, 1977.
\newblock Graduate Texts in Mathematics, No. 52.

\bibitem{hlmsk_homolog}
A.~Ya. Helemskii.
\newblock {\em The homology of {Banach} and topological algebras. {Transl}.
  from the {Russian} by {Alan} {West}}, volume~41 of {\em Math. Appl., Sov.
  Ser.}
\newblock Dordrecht etc.: Kluwer Academic Publishers, 1989.

\bibitem{hlmsk_fa}
A.~Ya. Helemskii.
\newblock {\em Lectures and exercises on functional analysis}, volume 233 of
  {\em Transl. Math. Monogr.}
\newblock Providence, RI: American Mathematical Society (AMS), 2006.

\bibitem{hm5}
Karl~H. Hofmann and Sidney~A. Morris.
\newblock {\em The structure of compact groups. {A} primer for the student. {A}
  handbook for the expert}, volume~25 of {\em De Gruyter Stud. Math.}
\newblock Berlin: De Gruyter, 5th edition edition, 2023.

\bibitem{zbMATH03539905}
Karl~Heinrich Hofmann.
\newblock Bundles and sheaves are equivalent in the category of {Banach}
  spaces.
\newblock {{\(K\)}}-{Theory} {Oper}. {Algebr}., {Proc}. {Conf}.
  {Athens}/{Georgia} 1975, {Lect}. {Notes} {Math}. 575, 53-69 (1977)., 1977.

\bibitem{hk_shv-bdl}
Karl~Heinrich Hofmann and Klaus Keimel.
\newblock Sheaf theoretical concepts in analysis: {Bundles} and sheaves of
  {Banach} spaces, {Banach} {C}({X})-modules.
\newblock Applications of sheaves, {Proc}. {Res}. {Symp}., {Durham} 1977,
  {Lect}. {Notes} {Math}. 753, 415-441 (1979)., 1979.

\bibitem{htt_d}
Ryoshi Hotta, Kiyoshi Takeuchi, and Toshiyuki Tanisaki.
\newblock {\em {{\(D\)}}-modules, perverse sheaves, and representation theory.
  {Translated} from the {Japanese} by {Kiyoshi} {Takeuchi}}, volume 236 of {\em
  Prog. Math.}
\newblock Basel: Birkh{\"a}user, expanded edition edition, 2008.

\bibitem{ks_shv-mfld}
Masaki Kashiwara and Pierre Schapira.
\newblock {\em Sheaves on manifolds. {With} a short history ``{Les} d{\'e}buts
  de la th{\'e}orie des faisceaux'' by {Christian} {Houzel}}, volume 292 of
  {\em Grundlehren Math. Wiss.}
\newblock Berlin etc.: Springer-Verlag, 1990.

\bibitem{k_tvs-2}
Gottfried K\"{o}the.
\newblock {\em Topological vector spaces. {II}}, volume 237 of {\em Grundlehren
  der Mathematischen Wissenschaften [Fundamental Principles of Mathematical
  Sciences]}.
\newblock Springer-Verlag, New York-Berlin, 1979.

\bibitem{lnc_hilb}
E.~C. Lance.
\newblock {\em Hilbert {$C^*$}-modules}, volume 210 of {\em London Mathematical
  Society Lecture Note Series}.
\newblock Cambridge University Press, Cambridge, 1995.
\newblock A toolkit for operator algebraists.

\bibitem{zbMATH06124312}
John Mather.
\newblock Notes on topological stability.
\newblock {\em Bull. Am. Math. Soc., New Ser.}, 49(4):475--506, 2012.

\bibitem{mnk}
James~R. Munkres.
\newblock {\em Topology}.
\newblock Prentice Hall, Inc., Upper Saddle River, NJ, 2000.
\newblock Second edition of [ MR0464128].

\bibitem{rud_rc_3e_1987}
Walter Rudin.
\newblock {\em Real and complex analysis}.
\newblock McGraw-Hill Book Co., New York, third edition, 1987.

\bibitem{ryan_ban}
Raymond~A. Ryan.
\newblock {\em Introduction to tensor products of {B}anach spaces}.
\newblock Springer Monographs in Mathematics. Springer-Verlag London, Ltd.,
  London, 2002.

\bibitem{scht_hndbk}
Eric Schechter.
\newblock {\em Handbook of analysis and its foundations}.
\newblock San Diego, CA: Academic Press, 1997.

\bibitem{ss_countertop}
Lynn~Arthur Steen and J.~Arthur Seebach, Jr.
\newblock {\em Counterexamples in topology}.
\newblock Dover Publications, Inc., Mineola, NY, 1995.
\newblock Reprint of the second (1978) edition.

\bibitem{wo}
N.~E. Wegge-Olsen.
\newblock {\em {$K$}-theory and {$C^*$}-algebras}.
\newblock Oxford Science Publications. The Clarendon Press, Oxford University
  Press, New York, 1993.
\newblock A friendly approach.

\bibitem{wil_top}
Stephen Willard.
\newblock {\em General topology}.
\newblock Dover Publications, Inc., Mineola, NY, 2004.
\newblock Reprint of the 1970 original [Addison-Wesley, Reading, MA;
  MR0264581].

\end{thebibliography}

\Addresses

\end{document}